\documentclass[10pt,reqno]{amsart}
\usepackage[colorlinks=true,
pdfstartview=FitV, linkcolor=blue, citecolor=green,
urlcolor=]{hyperref}
\usepackage{enumerate}
\usepackage{mathabx}
\usepackage{amssymb,amsfonts, amsmath}
\usepackage{hyperref}
\usepackage[title]{appendix}
\newtheorem{theorem}{Theorem}
\newtheorem{lemma}[theorem]{Lemma}
\newtheorem{proposition}[theorem]{Proposition}
\newtheorem{corollary}[theorem]{Corollary}

\newtheorem{claim}[theorem]{Claim}
\newtheorem{rmk}[theorem]{\normalfont{\em{Remark}}}

\renewcommand{\H}{\textup{H}}
\makeatletter
\renewcommand*\env@matrix[1][\arraystretch]{%
\edef\arraystretch{#1}%
\hskip -\arraycolsep
\let\@ifnextchar\new@ifnextchar
\array{*\c@MaxMatrixCols c}}
\makeatother
\AtEndDocument{\bigskip{\footnotesize%
\textsc{Max Planck Institute for Mathematics in the Sciences, Inselstra{\ss}e 22, 04103 Leipzig, Germany } \par  
 \textit{E-mail address}: \texttt{konstantinos.tsouvalas@mis.mpg.de} }}
\date{\today}
\title[Separability of unipotent-free abelian subgroups in linear groups]{Separability of unipotent-free abelian subgroups in linear groups}
\date{\today}

\author{Konstantinos Tsouvalas}

\begin{document}

\frenchspacing

\maketitle

\begin{abstract} Let ${\bf F}$ be a field of characteristic zero. It is proved that for any finitely generated linear group $\Gamma<\mathsf{GL}_n({\bf F})$, every unipotent-free abelian subgroup of $\Gamma$ is separable.\end{abstract}

\section{Introduction}

A subgroup $\H<\Gamma$ is called {\em separable} if for every $g\in \Gamma\smallsetminus \H$, there is a finite group $F$ and a homomorphism $\phi:\Gamma\rightarrow F$ with $\phi(g) \in F \smallsetminus \phi(\H)$. Equivalently, $\H<\Gamma$ is separable if it is an intersection of finite-index subgroups of $\Gamma$. Determining which subgroups in particular classes of finitely generated groups are separable has been a problem of broad interest in geometric group theory; we refer the reader to \cite{Formanek, Scott, ALR, Hamilton, Hsu-Wise,  Metaftsis-Raptis, Leininger-McReynolds, Wilton, Brid-Wil-sep, Fruchter} for a non exhaustive list of related results.

In the present note, we are interested in separability of abelian subgroups in linear groups. Let ${\bf F}$ be a field. Recall that a matrix $g\in \mathsf{GL}_n({\bf F})$ is {\em unipotent} if $g-\textup{I}_n$ is a nilpotent matrix. A subgroup $\H<\mathsf{GL}_n({\bf F})$ is called {\em unipotent-free} if it does not contain any non-trivial unipotent matrices. Our main result is the following.

 \begin{theorem}\label{main} Let ${\bf F}$ be a field of characteristic zero and let $\Gamma<\mathsf{GL}_n({\bf F})$, $n>0$, be a finitely generated subgroup. Then, any abelian unipotent-free subgroup of $\Gamma$ is separable. \end{theorem}

In Theorem \ref{main}, the assumption that the abelian subgroup of $\Gamma$ is unipotent-free is necessary, since abelian subgroups of linear groups, which contain unipotent elements, might fail to be separable (see Remark \ref{BS}). 

Crucial for the proof of Theorem \ref{main} is Proposition \ref{order}, which essentially follows by combining the main results from \cite{Chevalley} and \cite{Grunewald-Segal}. Theorem \ref{main} applies to unipotent-free finitely generated linear groups in characteristic zero, including subgroups of compact Lie groups, uniform lattices in algebraic semisimple Lie groups, quasi-isometrically embedded subgroups of $\mathsf{GL}_m(\mathbb{C})$, \hbox{$m\geq 2$,} all of whose cyclic subgroups are undistorted (including Anosov subgroups of $\mathsf{GL}_m(\mathbb{C})$ \cite{Labourie, GW}), and projective convex cocompact subgroups of $\mathsf{PGL}_{\ell}(\mathbb{R})$ \cite{DGK}.

\section{proof of theorem \ref{main}}
We devote the rest of the note for the proof of Theorem \ref{main}. For a commutative ring $A$ denote by $A^{\times}$ the multiplicative group of its units and let $\overline{\mathbb{Q}}$ be the field of algebraic numbers. For a number field $K\subset \overline{\mathbb{Q}}$ denote by $\mathcal{O}_{K}$ the ring of integers in $K$. For $q\in \mathbb{N}$ and $x\in K$, $x\equiv 1(\textup{mod}q)$ if $x=1+\kappa qz^{-1}$ for some $\kappa,z\in \mathcal{O}_K$ where $z,q$  are coprime in $\mathcal{O}_K$. 

\begin{proposition}\label{order} Let $A$ be a finitely generated subring of $\mathbb{C}$ and $r\in \mathbb{N}$. There exists a finite ring $\mathsf{R}$ and a homomorphism $\phi_r:A\rightarrow \mathsf{R}$ with the property that if $x\in A^{\times}$ and $\phi_r(x)=1$, then there is $y\in A^{\times}$ such that $x=y^r$.\end{proposition}

\begin{proof} The group of units $A^{\times}$ is finitely generated by \cite{SamuelP}. By the specialization theorem of Grunewald--Segal \cite[Theorem A]{Grunewald-Segal} there is a ring homomorphism $\theta:A\rightarrow \overline{\mathbb{Q}}$ such that the restriction of $\theta$ on $A^{\times}$ is faithful. Since $\theta(A)$ is finitely generated, there is $b\in \mathbb{N}$ such that $\theta(A)$ is contained in $\mathcal{O}_K[\frac{1}{b}]$, for some number field $K$. Given $r\in \mathbb{N}$, by applying Chevalley's theorem \cite[Théorème 1]{Chevalley} for the finitely generated abelian subgroup $\theta(A^{\times})$ of $\mathcal{O}_{K}[\frac{1}{b}]^{\times}$, there is $q\in \mathbb{N}$, relatively prime to $b$, such that whenever $x\in A^{\times}$ is an element with $\theta(x)\equiv 1 (\textup{mod}q)$, then there is $y\in A^{\times}$ such that $\theta(x)=\theta(y)^r$, or equivalently, $x=y^r$ as $\theta$ is injective restricted on $A^{\times}$. Let $(q)\subset \mathcal{O}_K[\frac{1}{b}]$ be the ideal generated by $q$. As $b,q\in \mathbb{N}$ are relatively prime, $\mathbb{Z}[\frac{1}{b}](q)/(q)$ is a finite ring of cardinality  at most equal to $q$. Now since $\mathcal{O}_{K}[\frac{1}{b}]$ is a finitely generated $\mathbb{Z}[\frac{1}{b}]$-module with a basis of $[K:\mathbb{Q}]$ elements, the quotient $\mathsf{R}:=\mathcal{O}_K[\frac{1}{b}]/(q)$ is a finite ring of cardinality at most equal to $q^{[K:\mathbb{Q}]}$.  Thus, if $\pi:\mathcal{O}_K[\frac{1}{b}]\rightarrow \mathsf{R}$ denotes the natural projection and $\pi(\theta(x))=1$ for some $x\in A^{\times}$, then $\theta(x)=1+\kappa qb^{-m}$ for some $m\in \mathbb{N}_{\geq 0}$, $\kappa \in \mathcal{O}_K$, and there is $y\in A^{\times}$ with $x=y^r$. This shows that the homomorphism $\phi_r:=\pi\circ \theta$ and the finite ring $\mathsf{R}$ satisfy the conclusion of the proposition.\end{proof}

We will also repeatedly use the following fact.

\begin{lemma}\textup{(\cite[Lemma 2.6]{Long-Reid})}\label{fi1} Let $\Gamma$ be a group, $\H<\Gamma$ a subgroup and $\H'<\H$ a finite-index subgroup. If $\H'<\Gamma$ is separable, then $\H<\Gamma$ is separable.\end{lemma}

Denote by $\mathsf{B}^{+}_{n}<\mathsf{GL}_n(\mathbb{C})$ the Borel subgroup of upper triangular matrices and by \hbox{$\mathsf{U}^{+}_n<\mathsf{B}^{+}$} the subgroup of unitriangular matrices. We shall prove the following characterization of separability for a solvable subgroup of a finitely generated subgroup of $\mathsf{GL}_n(\mathbb{C})$.

 \begin{theorem}\label{main2} Let $\Gamma<\mathsf{GL}_n(\mathbb{C})$, $n>0$, be a finitely generated group and $\mathsf{S}<\Gamma \cap \mathsf{B}_n^{+}$ a subgroup. Then $\mathsf{S}$ is separable in $\Gamma$ if and only if $\mathsf{S}\cap \mathsf{U}_n^{+}$ is separable in $\Gamma$. \end{theorem}

\begin{proof}First, note that if $\mathsf{S}<\Gamma\cap \mathsf{B}_n^{+}$ is separable in $\Gamma$, then, since $\Gamma \cap \mathsf{U}_n^{+}$ is separable in $\Gamma$ (see \cite[page 113]{Bergeron}), $\mathsf{S}\cap \mathsf{U}_n^{+}=\mathsf{S}\cap (\Gamma \cap \mathsf{U}_n^{+})$ is also separable in $\Gamma$.

Now we prove that if $\mathsf{S}^{+}:=\mathsf{S}\cap \mathsf{U}_n^{+}$ is separable in $\Gamma$ then $\mathsf{S}$ is separable. Let $A\subset \mathbb{C}$ be the subring generated by the entries of a finite generating subset of $\Gamma$ such that $\Gamma<\mathsf{GL}_n(A)$. Consider the homomorphism $\pi':\mathsf{B}_n^{+}\rightarrow (\mathbb{C}^{\times})^n$ sending an element $g\in \mathsf{B}_n^{+}$ to the vector of its diagonal entries such that $\pi'(\mathsf{S})$ is contained in the finitely generated abelian group $(A^{\times})^{n}$. Since $\mathsf{S}\cap \textup{ker}\pi' =\mathsf{S}^{+}$ it follows that $\mathsf{S}/\mathsf{S}^{+}$ is a finitely generated abelian group.

If $\mathsf{S}/\mathsf{S}^{+}$ is finite, then, as $\mathsf{S}^{+}$ is separable in $\Gamma$, $\mathsf{S}$ is also separable by Lemma \ref{fi1}. Henceforth, we assume that $\mathsf{S}/\mathsf{S}^{+}$ is infinite and we pass to a finite-index subgroup $\mathsf{S}_1<\mathsf{S}$, containing $\mathsf{S}^{+}$, such that $\mathsf{S}_1/\mathsf{S}^{+}$ is a torsion-free abelian group of rank equal to $q \in \mathbb{N}$.

\begin{claim}\label{claim} For every $m\in \{0,1,\ldots,q \}$ and $\Sigma\subset \mathsf{S}_1$ a finite subset such that $\langle \{w\mathsf{S}^{+}:w \in \Sigma\}\rangle$ is a free-abelian subgroup of $\mathsf{S}_1/\mathsf{S}^{+}$ of rank $m$, then $\langle \mathsf{S}^{+} \cup \Sigma \rangle$ is a separable subgroup of $\Gamma$.\end{claim}

\begin{proof}[Proof of Claim \ref{claim}] We use induction on $m\in \{0,\ldots, q\}$. 
If $m=0$, the claim holds true since by assumption $\langle \mathsf{S}^{+} \cup \Sigma \rangle=\mathsf{S}^{+}$ is separable in $\Gamma$. \par Let $1\leq m\leq q$ and assume that whenever $\Sigma$ is a subset of $\mathsf{S}_1$ such that $\langle \{\sigma \mathsf{S}^{+}:\sigma \in \Sigma\} \rangle$ is a free-abelian group of rank at most $m-1$, then $\langle \mathsf{S}^{+} \cup \Sigma \rangle$ is a separable in $\Gamma$. 

Suppose that $\Sigma'\subset \mathsf{S}_1$ is a finite subset such that $\langle \{\sigma\mathsf{S}^{+}:\sigma\in \Sigma'\}\rangle\cong \mathbb{Z}^m$ and choose a set of free generators $\{g_1\mathsf{S}^{+},\ldots,g_{m}\mathsf{S}^{+}\}$, $g_i\in \Sigma'$. Under the inductive hypothesis, we shall show that $\mathsf{S}_m^{+}:=\langle \mathsf{S}^{+}\cup \Sigma'\rangle=\langle \mathsf{S}^{+} \cup \{g_1,\ldots,g_m\}\rangle$ is separable in $\Gamma$. For for every $i\in \{1,\ldots,m\}$ write \begin{align}\label{g_i}g_i= \begin{pmatrix}[0.9]
\lambda_{1i} & \cdots &  \ast \\
  &  \ddots &  \vdots \\
  &    &\lambda_{ni} \\
\end{pmatrix}.\end{align} Since $g_{m}\mathsf{S}^{+}$ has infinite order in $\mathsf{S}_1/\mathsf{S}^{+}$, fix $s\in \{1,\ldots,n\}$ such that $\lambda_{ms}\in \mathbb{C}$ is not a root of unity. Let $p\in \mathbb{N}$ be the order of the torsion subgroup of $\langle \lambda_{1s},\ldots,\lambda_{ms}\rangle<A^{\times}$ and choose a non-empty subset $I\subset \{1,\ldots,m\}$ such that $\langle \lambda_{1s}^p,\ldots,\lambda_{ms}^p\rangle\cong \mathbb{Z}^{|I|}$ and \begin{align*}\mathsf{G}_{s}:=\langle \lambda_{1s}^p,\ldots,\lambda_{ms}^p\rangle=\langle \{\lambda_{is}^p: i\in I\}\rangle\cong \mathbb{Z}^{|I|}.\end{align*} 

From now own, we fix an ordering for the set $I$ (resp. $\{1,\ldots,m\}\smallsetminus I$) so that any product indexed by $I$ or its complement is with respect to this ordering.

Set $D:=[\langle \lambda_{1s},\ldots,\lambda_{ms}\rangle:\mathsf{G}_s]\leq p^m$ . For every $i\in I$, the $(s,s)$-entry of $g_i^p\in \mathsf{S}_1$ is equal to $z_{i}:=\lambda_{is}^p$. For every $t \notin I$, the $(s,s)$-entry of $g_{t}^{D}\in \mathsf{S}_1$ is equal to $\lambda_{ts}^D\in \mathsf{G}_s$. By writing $\lambda_{ts}^{D}=\prod_{i\in I}\lambda_{i}^{pt_i}$, for some $t_i\in \mathbb{Z}$, we choose $w_{t}:=\prod_{i\in I}g_{i}^{-pt_i}$ such that the $(s,s)$-entry of $w_tg_t^{D}\in \mathsf{S}_1$ is equal to $1$ for $t\notin I$. Let $\overline{g}_t\in \mathsf{S}_1$ be the element defined as follows: \begin{align}\label{gt-def}\overline{g}_t:=\left\{\begin{matrix}
g_{t}^{p}, & t \in I  \\
w_{t}g_{t}^D, & t \notin I  \\
\end{matrix}\right.,\end{align}  and consider the subgroup $\overline{\mathsf{S}}_{m}^{+}:=\langle \mathsf{S}^{+}\cup \{\overline{g}_1,\ldots,\overline{g}_m\}\rangle$ of $\mathsf{S}_1$. Since $\langle \overline{g}_1\mathsf{S}^{+},\ldots, \overline{g}_m\mathsf{S}^{+}\rangle$ is a subgroup of $\langle g_1\mathsf{S}^{+},\ldots, g_m\mathsf{S}^{+}\rangle$ of index at most equal to $p^{|I|}D^{m-|I|}$, \hbox{$\overline{\mathsf{S}}_m^{+}$ is of finite index in $\mathsf{S}_m^{+}$.}

We will show that $\overline{\mathsf{S}}_m^{+}$ is a separable subgroup of $\Gamma$; this is enough to conclude that $\mathsf{S}_m^{+}$ is separable in $\Gamma$ by Lemma \ref{fi1}. For this, note that $\Gamma \cap \mathsf{B}_n^{+}$ is a separable subgroup of $\Gamma$ (see \cite[page 113]{Bergeron}) which clearly contains $\overline{\mathsf{S}}_m^{+}$. Fix $h\in \Gamma\smallsetminus \overline{\mathsf{S}}_m^{+}$. If $h\in \Gamma\smallsetminus \Gamma \cap \mathsf{B}_n^{+}$, there is a finite-index subgroup $N<\Gamma$, containing $\Gamma \cap \mathsf{B}_n^{+}$, such that $h\in\Gamma \smallsetminus N$. Now suppose that $h\in \Gamma \cap \mathsf{B}_n^{+}$ and write $$h= \begin{pmatrix}[0.9]
\nu_1 & \cdots  &  \ast \\
     &  \ddots &  \vdots \\
 &    &\nu_{n} \\
\end{pmatrix}.$$ 

Let $A':=A[\nu_s,\nu_s^{-1}]$ be the subring of $\mathbb{C}$ generated by $A\cup\{\nu_s,\nu_s^{-1}\}$ such that $\Gamma<\mathsf{GL}_n(A')$. Then $\mathsf{G}_s=\langle \{z_i:i\in I\}\rangle$, $z_i=\lambda_{is}^p$, is clearly a subgroup of $(A')^{\times}$, hence, as $(A')^{\times}$ is finitely generated \cite{SamuelP}, we may choose a subgroup\footnote{Consider the quotient $(A')^{\times}/\mathsf{G}_s=\mathbb{Z}^c\times M$, where $M$ is finite and $c\geq 0$.  If $c=0$, we take $\mathsf{T}_s=\{1\}$. If $c>0$, choose $a_1,\ldots,a_{c}\in (A')^{\times}$ such that $\langle a_1\mathsf{G}_s,\ldots,a_{c}\mathsf{G}_s\rangle=\mathbb{Z}^{c}$ and $\mathsf{T}_s=\langle a_1,\ldots,a_{c}\rangle$.} $\mathsf{T}_{s}<(A')^{\times}$ such that $\mathsf{T}_{s}\cap \mathsf{G}_s=\{1\}$, $\mathsf{T}_s\mathsf{G}_{s}\cong \mathsf{T}_{s}\times \mathsf{G}_{s}$ and $\mathsf{T}_s\mathsf{G}_s<(A')^{\times}$ is of finite-index. Let $d:=[(A')^{\times}:\mathsf{T}_s\mathsf{G}_s]$ and consider a decomposition \begin{align}\label{abelian4}\nu_s^{d}=\kappa \prod_{i\in I}z_i^{q_i}\end{align} for some $\kappa \in \mathsf{T}_{s}$ and $q_i\in \mathbb{Z}$ for $i\in I$. 

Given $h\in \Gamma \cap \mathsf{B}^{+}$, we will exhibit a group homomorphism of $\phi:\Gamma\rightarrow F$ onto a finite group $F$ such that $\phi(h)\in F\smallsetminus \phi(\overline{\mathsf{S}}_m^{+})$. There are two cases to consider:

\smallskip

\noindent {\em Case 1.} {\em There is $j \in I$ such that $d$ does not divide $q_j$}.
\smallskip

Suppose there is such $j\in I$ so that $q_j\neq 0$. By Proposition \ref{order} there is a ring homomorphism $\phi_{d}:A'\rightarrow \mathsf{R}$, where $\mathsf{R}$ is a finite ring, such that if $z\in (A')^{\times}$ and $\phi_{d}(z)=1$, then $z$ is a $d$-power of an element in $(A')^{\times}$. Let $(\phi_{d})_{n}:\mathsf{GL}_{n}(A')\rightarrow \mathsf{GL}_{n}(\mathsf{R})$ be the induced homomorphism by applying $\phi_d$ entrywise. Then $(\phi_d)_{n}(h)\notin  (\phi_d)_{n}(\overline{\mathsf{S}}_m^{+})$. If not, since $\mathsf{S}^{+}$ contains the commutator subgroup $[\mathsf{S}_1,\mathsf{S}_1]$ of $\mathsf{S}_1$, we may find $w_0\in \mathsf{S}^{+}$ and integers $k_1,\ldots,k_m\in \mathbb{Z}$ such that $$(\phi_d)_{n}(h)=(\phi_{d})_{n}\bigg(w_0\prod_{i\in I}\overline{g}_i^{k_i}\prod_{t\notin I}\overline{g}_t^{k_t}\bigg).$$ The $(s,s)$-entry of $w_0$ is equal to 1 and recall from (\ref{gt-def}) that  the $(s,s)$-entries of $\overline{g}_t$, $t\notin I$, is $1$ and the $(s,s)$-entry of $\overline{g}_i^{k_i}$ is $z_i^{k_i}$ for $i\in I$. Thus, we obtain $\phi_{d}(\nu_s)=\phi_{d}\big(\prod_{i\in I}z_{i}^{k_i}\big)$. In particular, by the choice of $\phi_d$, there is $y\in (A')^{\times}$ such that $$\nu_s=y^d \prod_{i\in I}z_{i}^{k_i}.$$ Since $y^{d}\in \mathsf{G}_s\mathsf{T}_s=\langle\{z_i:i\in I\}\rangle \mathsf{T}_s$, write $y^d=t_0\prod_{i\in I}z_i^{\beta_i}$ for some $t_0\in \mathsf{T}_s$, $\beta_i\in \mathbb{Z}$. By (\ref{abelian4}) we have that, $$\nu_s^d=\kappa \prod_{i\in I}z_i^{q_i}=y^{d^2}\prod_{i\in I}z_{i}^{dk_i}=t_0^d \prod_{i\in I}z_{i}^{dk_i+d\beta_i}.$$ As $\mathsf{G}_s$ is a free-abelian group with basis $\{z_i:i\in I\}$ and $\mathsf{G}_s\cap \mathsf{T}_s=\{1\}$, we necessarily have $q_j=dk_j+d\beta_{j}$ for every $j\in I$, contradicting the assumption that $d$ does not divide $q_k$. This completes the proof of this case and we take $\phi:=(\phi_d)_{n}$.

\smallskip

\noindent {\em Case 2.} {\em  For every $j \in I$, $q_j$ is a multiple of $d$}.
\smallskip

We consider the element \begin{align}\label{h-bar}\overline{h}:=\bigg(\prod_{i\in I}\overline{g}_{i}^{-\frac{q_i}{d}}\bigg)h.\end{align} Since $h\in\Gamma \smallsetminus \overline{\mathsf{S}}_m^{+}$, clearly $\overline{h}\in\Gamma \smallsetminus \overline{\mathsf{S}}_m^{+}$. Note that the subgroup $\langle \{\overline{g}_t\mathsf{S}^{+}:t\notin I\}\rangle$ of $\mathsf{S}_1/\mathsf{S}^{+}$ is free-abelian of rank at most $m-|I|\leq m-1$, thus, by the inductive step, $\mathsf{S}_{I}^{+}:=\langle \mathsf{S}^{+}\cup \{\overline{g}_t:t\notin I\}\rangle$ is a separable subgroup of $\Gamma$. Since $\mathsf{S}_{I}^{+}<\overline{\mathsf{S}}_m^{+}$ and $h\in \Gamma\smallsetminus \overline{\mathsf{S}}_m^{+}$, there is a homomorphism $\varphi_1:\Gamma \rightarrow F_1$, onto a finite group $F_1$, such that $\varphi_1\big(\overline{h}\big)\in F_1\smallsetminus \varphi_1(\mathsf{S}_{I}^{+})$.

Set $\ell:=d |F_1|$ and by Proposition \ref{order} fix a ring homomorphism $\phi_{\ell}:A'\rightarrow \mathsf{R}'$, where $\mathsf{R}'$ is a finite ring, so that $\phi_{\ell}$ satisfies the conclusion of Proposition \ref{order} for $r:=\ell$. Consider the product homomorphism $\phi:\Gamma \rightarrow  \mathsf{GL}_{n}(\mathsf{R}')\times F_1$, $\phi:=(\phi_{\ell})_{n} \times \varphi_1$.

We claim that $\phi(\overline{h})\notin \phi(\overline{\mathsf{S}}_m^{+})$. If not, since $[\mathsf{S}_1,\mathsf{S}_1]$ is contained in $\overline{\mathsf{S}}_m^{+}$, there are $w_1\in \mathsf{S}^{+}$ and $k_1,\ldots,k_m\in \mathbb{Z}$ with \begin{align}\label{abelian6}(\phi_\ell)_{n}\big(\overline{h}\big)&=(\phi_{\ell})_{n}\bigg(w_1\prod_{i\in I}\overline{g}_i^{k_i}\prod_{t\notin I}\overline{g}_t^{k_t}\bigg),\\  \label{abelian5} \varphi_1\big(\overline{h}\big)&=\varphi_1\bigg(w_1\prod_{i\in I}\overline{g}_i^{k_i}\prod_{t\notin I}\overline{g}_t^{k_t}\bigg).\end{align}
By (\ref{h-bar}), the $(s,s)$-entry of $\overline{h}\in \Gamma$ is equal to $\nu_s\prod_{i\in I}z_{i}^{-\frac{q_i}{d}}$, hence, by looking at the $(s,s)$-entries of both sides in (\ref{abelian6}) we have $\phi_{\ell}\big(\nu_s\prod_{i\in I}z_{i}^{-\frac{q_i}{d}}\big)=\phi_{\ell}\big(\prod_{i\in I}z_{i}^{k_i}\big).$ In particular, by the choice of $\phi_{\ell}$, we may find $\omega\in (A')^{\times}$ such that \begin{align}\label{abelian4'}\nu_s\prod_{i\in I}z_{i}^{-\frac{q_i}{d}}=\omega^{\ell}\prod_{i\in I}z_{i}^{k_i}.\end{align} To this end, as $\omega^d\in \mathsf{T}_s\mathsf{G}_s$, write $\omega^d=t_1\prod_{i\in I}z_{i}^{c_i}$ for some $t_1\in \mathsf{T}_s$, $c_i\in \mathbb{Z}$. By raising both parts of (\ref{abelian4'}) to the $d$-th power and  using (\ref{abelian4}) we conclude that $$\kappa=\nu_s^d\prod_{i\in I}z_i^{-q_i}=t_1^{\ell} \prod_{i\in I}z_{i}^{dk_i+\ell c_i}\in \mathsf{T}_s.$$ Again, as $\mathsf{G}_s$ is free-abelian with basis $\{z_i:i\in I\}$ and $\mathsf{G}_s\cap \mathsf{T}_s=\{1\}$, this implies $dk_i+\ell c_i=0$, or equivalently, $k_i=-\frac{\ell}{d}|F_1|=-c_i|F_1|$ for every $i\in I$. By using (\ref{abelian5}), we see that $$\varphi_1(\overline{h})=\varphi_1(w_1)\prod_{i\in I}\varphi_1(g_i)^{-c_i|F_1|}\prod_{t\notin I}\varphi_1(\overline{g}_t^{k_t})=\varphi_1\bigg(w_1\prod_{t\notin I}\overline{g}_t^{k_t}\bigg),$$ contradicting the fact that $\varphi_1(\overline{h})\in F_1\smallsetminus \varphi_1(\mathsf{S}_{I}^{+})$.
\par Finally, we conclude that $\phi(\overline{h}),\phi(h)\in (\mathsf{GL}_{n}(\mathsf{R}')\times F_1) \smallsetminus \phi(\overline{\mathsf{S}}_m^{+})$ for $\phi:=(\phi_{\ell})_n\times \varphi_1$. This  finishes the proof of Case 2.\end{proof}

Therefore, $\overline{\mathsf{S}}_m^{+}<\Gamma$ is separable and since $\overline{\mathsf{S}}_m^{+}$ has finite index in $\mathsf{S}_m^{+}$, we conclude that $\mathsf{S}_{m}^{+}$ is separable in $\Gamma$. This finishes the proof of the induction and Claim \ref{claim} follows. In particular, we deduce that $\mathsf{S}_1<\Gamma$ is separable. As $\mathsf{S}_1$ has finite index in $\mathsf{S}$, it follows that $\mathsf{S}<\Gamma$ is separable.\end{proof}

\begin{proof}[Proof of Theorem \ref{main}] As $\Gamma$ is finitely generated we may assume that ${\bf F}=\mathbb{C}$. Let $\H<\Gamma$ be an abelian unipotent-free subgroup of $\Gamma$. By the Lie--Kolchin theorem \cite[Theorem 17.6]{Humphreys}, up to conjugating $\Gamma$ by an element  of $\mathsf{GL}_n(\mathbb{C})$, there is a finite-index subgoup $\H_1<\H$ contained in $ \mathsf{B}_n^{+}$. Since $\H_1\cap \mathsf{U}_n^{+}=\{1\}$ is separable in $\Gamma$, Theorem \ref{main2} implies that $\H_1$ is separable in $\Gamma$. It follows by Lemma \ref{fi1} that $\H$ is separable in $\Gamma$. \end{proof}

\begin{rmk}\label{BS} \normalfont{The assumption that the abelian subgroup $\H<\Gamma$ is unipotent-free in the proof of Theorem \ref{main} cannot be dropped as the following well-known example shows. Consider the Baumslag--Solitar group $\mathsf{BS}(1,2)=\langle t,a \ | \ tat^{-1}a^{-2}\rangle$ realized as a subgroup of $\mathsf{GL}_2(\mathbb{Z}[\frac{1}{2}])$ via the embedding $$t\mapsto \begin{pmatrix} 2 & 0  \\  0  & 1 \end{pmatrix}, \ a\mapsto \begin{pmatrix} 1 & 1 \\ 0  & 1 \end{pmatrix}.$$ The subgroup $\langle a \rangle<\mathsf{BS}(1,2)$ is not separable since for any finite group $F$ and a homomorphism $\tau:\mathsf{BS}(1,2)\rightarrow F$ with $\tau(a)\neq 1$, $\tau(a)$ has odd order}.\end{rmk}

While we apply Theorem \ref{main2} above for unipotent-free abelian groups, as a corollary we also deduce separability of solvable subgroups in a linear group containing a maximal unipotent subgroup.

\begin{corollary} Let $\Gamma$ be a finitely generated subgroup of $\mathsf{GL}_n(\mathbb{C})$, $n>0$. Any solvable subgroup of $\Gamma$, containing a maximal unipotent subgroup of $\Gamma$, is separable.\end{corollary}

\begin{proof} Let $\mathsf{U}<\Gamma$ be a maximal unipotent subgroup contained in a solvable subgroup $\mathsf{S}<\Gamma$. Up to conjugating $\Gamma$ by an element of $\mathsf{GL}_n(\mathbb{C})$, we may assume that $\mathsf{U}=\Gamma \cap \mathsf{U}_{n}^{+}$, hence $\mathsf{U}$ is separable in $\Gamma$ (see \cite[page 113]{Bergeron}). Since $\mathsf{S}$ is solvable, there is $g\in \mathsf{GL}_n(\mathbb{C})$ and a finite-index subgroup $\mathsf{S}_0<\mathsf{S}$ such that $\mathsf{S}_0<g\mathsf{B}_{n}^{+}g^{-1}$ and $\mathsf{S}_0\cap \mathsf{U}<g\mathsf{U}_{n}^{+}g^{-1}$. In addition, $\mathsf{S}_0\cap \mathsf{U}$ is of finite-index in $\mathsf{U}$, and since the Zariski closure of $\mathsf{U}$ in $\mathsf{GL}_n(\mathbb{C})$ is connected and $\mathsf{U}<\Gamma$ is maximal unipotent, we conclude that $\mathsf{U}<g\mathsf{U}_{n}^{+}g^{-1}$ and $\mathsf{U}=\Gamma\cap g\mathsf{U}_n^{+}g^{-1}$. Therefore, the finite-index subgroup $\mathsf{S}'=\mathsf{S}_0\mathsf{U}$ of $\mathsf{S}$ is contained in $g\mathsf{B}_n^{+}g^{-1}$ and the intersection $\mathsf{S}'\cap g\mathsf{U}_n^{+}g^{-1}=\Gamma\cap g\mathsf{U}_n^{+}g^{-1}=\mathsf{U}$ is separable in $\Gamma$. Thus, by Theorem \ref{main} we conclude that $\mathsf{S}'$ (and hence $\mathsf{S}$) is separable in $\Gamma$.\end{proof}

\subsection*{Acknowledgements} I became aware of the results \cite{Chevalley, Grunewald-Segal} from the preprint \cite{Farb-Alperin}. I would like to thank Subhadip Dey and Benson Farb for comments on a previous draft of this note.

\bibliographystyle{alpha}

\bibliography{biblio.bib}
\end{document}